\documentclass[11pt]{article}
\usepackage[a4paper]{anysize}\marginsize{3.1cm}{3.1cm}{0.9cm}{3cm}
   \pdfpagewidth=\paperwidth \pdfpageheight=\paperheight
\usepackage[latin1]{inputenc}
\usepackage{amssymb,amsmath}
\usepackage{graphicx}
\usepackage{enumerate}
\makeatletter
\renewcommand\@seccntformat[1]{\csname the#1\endcsname.\enspace}
\makeatother
\renewenvironment{abstract}{\begin{quote}\hrulefill\par\footnotesize\textbf{\abstractname.}}{\par\vskip-0.5\baselineskip\hrulefill\end{quote}}

\newtheorem{introtheorem}{Theorem}  
\newtheorem{introcor}[introtheorem]{Corollary}

\newtheorem{thm}{Theorem}[section]
\newtheorem{theorem}[thm]{Theorem}

\newtheorem{proposition}[thm]{Proposition}
\newtheorem{cor}[thm]{Corollary}

\newcommand\mkrmthm[2]{\newenvironment{#1}{\begin{#2}\rm}{\end{#2}}}
 \mkrmthm{definition}{tdefinition}
         \mkrmthm{remark}{tremark}
       \mkrmthm{example}{texample}
     \mkrmthm{question}{tquestion}

\newenvironment{proof}[1][Proof]{\trivlist\item[\hskip\labelsep{\textit{#1.}}]}{\hspace*{\fill}$\Box$\endtrivlist}

\renewcommand\emptyset{\varnothing}  
\renewcommand\ge{\geqslant}  
\renewcommand\le{\leqslant}  

\renewcommand\phi{\varphi}

\newcommand\grant[1]{{\renewcommand\thefootnote{}\footnotetext{#1.}}}
\newcommand\keywords[1]{{\renewcommand\thefootnote{}\footnotetext{\textit{Keywords:} #1.}}}
\newcommand\subclass[1]{{\renewcommand\thefootnote{}\footnotetext{\textit{Mathematics Subject Classification (2010):} #1.}}}

\newcommand\C{\mathbb C}
\newcommand\Q{\mathbb Q}

\newcommand\R{\mathbb R}
\newcommand\Z{\mathbb Z}
\renewcommand\P{\mathbb P}
\renewcommand\O{\mathcal O}
\newcommand\T{\mathcal T}
\renewcommand\v{ {\bf v}}
\newcommand\w{ {\bf w}}
\newcommand\uu{ {\bf u}}

\newcommand\be{$$\everymath{\displaystyle}\renewcommand\arraystretch{1.2}\begin{array}{rcl}}
\newcommand\ee{\end{array}$$}

\newcommand\compact{\itemsep=0cm \parskip=0cm}

\newcommand\set[1]{\left\{#1\right\}}

\renewcommand\to{\longrightarrow}

\newcommand\newop[2]{\newcommand#1{\mathop{\rm #2}\nolimits}}

\newop\End{End}
\newop\NE{\overline{NE}}
\newop\Amp{Amp}
\newop\Nef{Nef}
\newop\vol{vol}
\newop\pos{pos}
\newop\Eff{\overline{Eff}}
\newop\mBig{Big}

\newop\NS{NS}

\newop\id{id}
\newop\mult{mult}

\makeatletter
\newcommand\opteq[1]{\mathrel{\mathpalette\opt@eq{#1}}}
\newcommand{\opt@eq}[2]{%
  \begingroup
  \sbox\z@{$#1#2$}%
  \sbox\tw@{\resizebox{!}{.5\ht\z@}{$\m@th#1($}}%
  \nonscript\hskip-\wd\tw@
  \mkern1mu
  \raisebox{-.35\ht\z@}[0pt][0pt]{\resizebox{!}{.5\ht\z@}{$\m@th#1($}}%
  \mkern-1mu
  {#2}%
  \mkern-1mu
  \raisebox{-.35\ht\z@}[0pt][0pt]{\resizebox{!}{.5\ht\z@}{$\m@th#1)$}}%
  \mkern1mu
  \nonscript\hskip-\wd\tw@
  \endgroup
}
\makeatother

\newcommand{\geoq}{\opteq{\geq}}

\begin{document}

   \title{On exterior powers of the tangent bundle on toric varieties}
   \author{ David Schmitz}
   \date{\normalsize \today}
   \maketitle
   \thispagestyle{empty}
   \grant{The author was supported by DFG grant PE 305/13-1, ``Zur Positivit\" at in der komplexen Geometrie''}
   \subclass{14J60, 14M25}

\begin{abstract}
  We study the positivity of exterior powers of the tangent sheaf on toric varieties in order to 
  generalize results by Campana and Peternell about 3-folds with nef second exterior power 
  of the tangent bundle. Using the theory of equivariant vector bundles and the toric MMP, we 
  establish in the smooth case a criterion for the positivity of $\Lambda^m\T_X$ in terms of wall 
  relations. As an application, we  classify smooth  toric varieties of arbitrary dimension $n\ge3$ with $
  \Lambda^2\T_X$ nef and those  of dimension $n\ge 4$ with $\Lambda^3\T_X$ ample.   
\end{abstract}



\section*{Introduction}
The study of consequences of the positivity of the tangent bundle $\T_X$  
 of a projective manifold $X$  and of related bundles has been long and fruitful. The most famous instance is  Mori's result (\cite[Theorem 8]{M}) stating that ampleness of $\T_X$ forces an $n$-dimensional projective manifold $X$ to be isomorphic to $\P^n$. Campana and Peternell in \cite{CP91} weakened the 
 assumption of ampleness to nefness of $\T_X$, meaning that the tautological line bundle $
 \O_{\P(\T_X)}(1)$ is nef on $\P(\T_X)$. They classified the 3-folds with this property and  formulated 
 their well-known conjecture predicting that any Fano manifold with nef tangent bundle should be 
 rational homogeneous. This prediction is in accordance with Mok's earlier classification of compact 
 manifolds with a Kähler metric with non-negative bisectional curvature, yet this condition is much 
 stronger than nefness of $\T_X$.  We refer the reader to  \cite{Wis} for a survey of the progress made 
 on the Campana-Peternell conjecture. 
 
Instead of the tangent bundle itself, it is natural to study its exterior powers. The leading example 
 being $\Lambda^n\T_X=-K_X$, whose nefness forces the Kodaira dimension of $X$ to be 
at most 0 but yields very little in terms of classification. On the other hand, Campana and Peternell found the following classification of projective 3-folds with $\Lambda^2\T_X$ nef.
\begin{introtheorem}[\cite{CP94}]
	Let $X$ be a  projective 3-fold with $\Lambda^2\T_X$ nef. Then either 
	$\T_X$ is nef or $X$ is one of the following.
	\begin{enumerate}[a)]
		\item $X$ is the blowup of $\P^3$ in one point, or
		\item $X$ is a Fano 3-fold of index 2 and $b_2=1$, but not a double cover of the Veronese cone.  
	\end{enumerate}	
\end{introtheorem}
In the present note, we investigate positivity of arbitrary exterior powers of the tangent sheaf in 
the case of toric varieties. On the one hand, the focus on toric varieties is a substantial restriction of the 
subject matter. For example, in the above theorem, none of the 3-folds of case b) are toric. On the 
other hand,  the theory of toric varieties provides us with sufficient machinery to handle arbitrary 
dimensions. 

A particularly pleasant feature of (complete) toric varieties for our investigation is the fact that nefness and ampleness of an equivariant vector bundle $\mathcal F$ test on torus-invariant curves as has been shown in \cite{hmp}. As such a curve is isomorphic to $\P^1$, the restriction of $\mathcal F$ decomposes as a sum of line bundles, and the positivity of $\mathcal F$ is determined by the splitting types
$$
	\mathcal F|_C \cong \O_{\P^1}(a_1) \oplus \O_{\P^1}(a_2)\oplus\dots\oplus  \O_{\P^1}(a_{r})
$$
for invariant curves $C$. By utilizing the  \emph{Klyachko filtration} of $\mathcal F$, in \cite{hmp} the 
authors also provide an approach to determining these splitting types, which we will describe in Section
 \ref{s:equivariantBundles}.

By applying the above approach to the tangent bundle, we establish the following close relationship between positivity questions and the types of contractions arising in the toric minimal model program. The splitting type of $\T_X$ restricted to an invariant curve $C$ turns out to be given by the coefficients of the \emph{wall relation} $b_1\v_1+\dots+b_{n-1}\v_{n-1}+\v_n +\v_{n+1}=0$ corresponding to $C$. Here $\v_1,\dots,\v_{n-1}$ are the primitive generators of the rays spanning the wall $\tau\in\Sigma(n-1)$ corresponding to $C$ and $\v_n$ and $\v_{n+1}$ are the primitive generators of 
those rays which together with $\tau$ respectively span those two maximal cones $\sigma,\sigma'\in\Sigma(n)$, which intersect in $\tau$. Fixing this notation, the first main result of this note are the following criteria for positivity of exterior powers of $\T_X$.

\begin{introtheorem}
	 Let $X_\Sigma$ be a smooth toric variety of dimension $n$. Then for $1\le m\le n-1$ the exterior power 
	 $\Lambda^m\T_X$ is ample (nef) if and only if
	 for any $\tau\in\Sigma(n-1)$ with wall relation $b_1\v_1+\dots+b_{n-1}\v_{n-1}+\v_n +\v_{n+1}=0$ the inequalities
	 \begin{align*}
	 	b_{i_1}+b_{i_2}+\dots + b_{i_m} &\geoq 0\\
		2 + b_{j_1}+b_{j_2}+\dots + b_{j_{m-1}} &\geoq 0
	 \end{align*}
	 hold for each $1\le i_1<\dots<i_m\le n$ and $1\le j_1<\dots<i_{m-1} \le n$.\\
	 Similarly,  $\Lambda^n\T_X=-K_X$ is ample (nef) if and only if for any $\tau\in\Sigma(n-1)$ the 	coefficients  in the corresponding wall relation satisfy
	 $$
	 	2 + b_{1}+b_{2}+\dots + b_{{n-1}} \geoq 0.
	 $$
\end{introtheorem}
Let us note here that this theorem has the following rather surprising consequence.
\begin{introcor}
	Let $X$ be a smooth toric variety of dimension $n$. If $\Lambda^m\T_X$ is nef for some $1\le m<n$, then $X$ is Fano.
\end{introcor} 
Remembering that the signs of the $b_i$ in a wall relation corresponding to an invariant 
curve $C$ whose class is contained in an extremal ray of the Mori cone $\NE(X)$ determine the type of contraction given by the extremal ray of the class of $C$ (see section \ref{s:contractions}), the inequalities in the above theorem restrict the possible types of contractions, thus enabling us to use an inductive approach, at least in the case of small numbers $m$. For example, 
we readily see that the tangent bundle $\T_X$ itself is nef if and only if all $b_i$ in the wall relations corresponding to any invariant curve are positive, which in fact is equivalent to all contractions of $X$ being of fiber type. Similarly, the nefness of $\Lambda^2\T_X$ excludes flipping contractions as well as divisorial contractions for which the image of the exceptional divisor is positive dimensional. 
In fact, by using upper bounds on the intersections $-K_X\cdot C$ for extremal invariant curves $C$ recently established by Fujino and Sato (\cite{FS2}), we can also exclude divisorial contractions with 
singular image. Similar arguments work for the case when $\Lambda^3\T_X$ is ample. 

By the above analysis of possible contractions, we are enabled to inductively proving the 
following classification results the first of which in the toric situation extends Campana's and Peternell's classification to arbitrary dimensions.
\begin{introtheorem}
	Let $X$ be a smooth toric variety of dimension $n\ge3$ with $\Lambda^2\T_X$ nef. Then either $\T_X$ is nef, or 
	$X$ is the blowup of $P^n$ in a point. 
\end{introtheorem}

\begin{introtheorem}
	Let $X$ be a smooth toric variety of dimension $n\ge4$ with $\Lambda^3\T_X$ ample. 
	Then either $\T_X$ is nef, or 
	$X$ is the blowup of $P^n$ in a point. 
\end{introtheorem}


\section{Positivity of equivariant bundles}\label{s:equivariantBundles}
Our aim is to study exterior powers $\mathcal E$ of the tangent sheaf of a complete toric variety $X$. Such a sheaf
comes equipped with a natural  torus action as follows. For any $g\in T$, considered as an automorphism of
$X$, we have an isomorphism $\psi_g:g^\ast\mathcal E\cong\mathcal E$ such that for any two $g,g'\in T$ the 
obvious diagram commutes. Such sheaves are called $T$-equivariant, or simply toric sheaves. 

Note that the tangent sheaf $\T_X$ is reflexive by definition as the dual of the sheaf $\widehat{\Omega}^1_X$
of Zariski differentials. 
It turns out that reflexive equivariant sheaves on complete toric varieties can be specified by giving for any 
invariant prime divisor $D_\rho$ corresponding to a ray $\rho\in \Sigma(1)$ a filtration
$$
	E \supseteq \dots\supseteq E^\rho(i-1)\supseteq E^\rho(i)\supseteq E^\rho(i+1)\supseteq\dots\supseteq 0
$$
of a fixed finite dimensional vector space $E$. Concretely,  for a cone $\sigma\in \Sigma$ one 
considers the $k[\sigma]$-module $\Gamma(U_\sigma,\mathcal E)$ on the affine toric variety $U_\sigma$. 
It is equipped with a natural $M$-grading, and one writes $E^\sigma_m$ for the $\chi^m$-isotypical component. 
Since $\mathcal E$ is reflexive and $X$ is normal, sections are determined by their restriction to $X\setminus Y$ 
for any closed subset $Y$ of codimension at least two. In particular, 
 $\Gamma(U_\sigma,\mathcal E)=\bigcap_{\rho\in\Sigma(1)}\Gamma(U_\rho,\mathcal E)$ and the same description holds for 
 the graded pieces, thus $E^\sigma_m=\bigcap_{\rho\in\Sigma(1)} E^\rho_m$. Now, if $m-m'\in\rho^\perp$ 
 then $E^\rho_m=E^\rho_{m'}$. This leads to a filtration as above by setting
 $$
 	E^\rho((m\cdot\v_\rho)) = E^\rho_m
 $$
for the primitive generator $\v_\rho$ of the ray $\rho$. In case $\mathcal E$ is locally free, these 
filtrations satisfy the following compatibility condition: for any cone $\sigma\in\Sigma$ the line bundle $\mathcal E$ 
is trivial on the affine variety $U_\sigma$ so that $\Gamma(U_\sigma,\mathcal E)$ decomposes into a direct sum of free $k[\sigma]$-modules of rank 1. Klyachko proved in \cite{K} the following equivalence of categories: 

\begin{theorem}\label{Th:Filtration}
	The category of equivariant locally free sheaves on a complete toric variety is naturally equivalent 
	to the category of finite dimensional $k$-vector spaces $E$ together with a decreasing filtration $E^\rho(i)$ for each 
	ray $\rho\in\Sigma(1)$ such that for any maximal cone $\sigma\in\Sigma$, there is a decomposition 
	into $T$-eigenspaces $E=\bigoplus E_m$ such that 
	$$
		E^\rho(i)=\sum_{(m\cdot \v_\rho)\ge i} E_m
	$$
	for every ray $\rho\in\sigma(1)$. 
\end{theorem}
For the details of the above description we refer the reader to \cite{K}, or, for the more general 
framework of coherent sheaves, \cite{Per}.

Hering, Musta\c t\u a, and Payne in \cite{hmp} and subsequently Di Rocco, Jabbusch, and Smith in \cite{DJS} used Theorem 
\ref{Th:Filtration}  in order to study positivity of equivariant vector bundles on complete toric varieties. The fundamental result
for this approach  consists in the observation that nefness and ampleness of such bundles test on \emph{invariant} curves:

\begin{theorem}[\cite{hmp}, Theorem 2.1] \label{Th:TestOnCurves}
An equivariant vector bundle on a complete toric variety is nef (ample) 
if and only if its restriction to every torus-invariant curve is nef (ample).
\end{theorem}
Since a torus-invariant curve $C$ is isomorphic to $\P^1$, the restriction $\mathcal E|_C$ of a vector bundle $\mathcal E$ 
of rank $r$ decomposes as the direct sum of line bundles $\bigoplus_{i=1}^r \O_{\P^1}(a_i)$. In fact, for an equivariant vector bundle this is an  \emph{equivariant} splitting by \cite{Kum}. Now, the  positivity of the restriction 
can easily be read off the splitting type $(a_1,\dots,a_r)$. In order to determine the splitting types of restrictions of an equivariant 
bundle $\mathcal E$, Hering, Musta\c t\u a, Payne consider the restriction of $\mathcal E$ to $U_\sigma$ and $U_{\sigma'}$ where 
$\sigma$ and $\sigma'$ denote the two full dimensional cones in $\Sigma$ whose intersection is the wall $\tau$ corresponding to the 
invariant curve $C$. 

For any cone $\sigma\in \Sigma$, line bundles on the affine toric variety $U_\sigma$ are of the form 
$\mathcal L_\uu=\O_{U_\sigma}\mbox{div}(\chi^\uu)$ for some $m\in M$ and the restriction of an equivariant vector bundle to $U_\sigma$
decomposes as a  sum $\bigoplus \mathcal L_{\uu_i}$ of such line bundles. Furthermore, for a full-dimensional cone $\sigma$, 
the lattice points $\uu_i$ coming up in the splitting only depend on $\sigma $ and $\mathcal E$ (see \cite{hmp}). 

If two lattice points  $\uu$ and $\uu'$ agree as linear functionals on the wall $\tau$ then the equivariant line 
bundles $\mathcal L_\uu$ and $\mathcal L_{\uu'}$ glue together by the transition function $\chi^{\uu-\uu'}$ to form an equivariant line 
bundle $\mathcal L_{\uu,\uu'}$ on the union $U_\sigma\cup U_{\sigma'}$. Now, $\mathcal L_{\uu,\uu'}$ restricts to $C$ as follows:
Let $\w_\tau\in M$ be the primitive generator of $\tau^\bot$ which is positive on $\sigma$. Then $\uu-\uu'=m\w_\tau$ 
for some integer $m$ and we have $\mathcal  L_{\uu,\uu'}|_C\cong \O_{\P^1}(m)$.

Using continuous interpolations of Klyachko's filtration, which were introduced in \cite{pa}, it is shown in \cite{hmp} that by 
subdividing $\Sigma$, or equivalently, by passing to a higher toric birational model $\pi:X'\to X$, one can guarantee that 
the restrictions of $\pi^\ast\mathcal E$ to unions $U_{\sigma}\cup U{\sigma'}$  split equivariantly as sums of line bundles of the 
form  $\mathcal L_{\uu,\uu'}$. Then for any invariant curve $C'$ on $X'$ mapped isomorphically to $C$, the splitting types of $\pi^\ast
\mathcal E|{C'}$ and $\mathcal E|_C$ agree, so that we get up to reordering uniquely determined pairs of characters $(\uu_i,\uu_i')$ 
such that $\mathcal E|_C$ splits equivariantly as
$$
	\mathcal E|_C \cong \mathcal L_{\uu_1,\uu_1'}\oplus\dots\oplus \mathcal L_{\uu_r,\uu_r'}.
$$
Hence, the positivity of $\mathcal E$ is decided by determining the pairs $(\uu_i,\uu_i')$ of characters for each invariant curve $C$.
\cite{DJS} provides  a nice description of these pairs as well as the numbers $m$ in terms of collections (or parliaments) of 
certain polytopes reflecting the global sections of $\mathcal E$. We will not go into detail here, as for the bundles in question in 
this note the pairings of characters can readily be read off from the Klyachko filtrations.

\section{Toric extremal contractions}\label{s:contractions}
We collect here results from the theory of toric varieties, in particular  their birational geometry which we will need in the remainder 
of this note. The main reference is the comprehensive book  \cite{CLS} by Cox, Little, and Schenck; we follow their notation throughout. 

Let $X$ be  an  $n$-dimensional toric variety corresponding to a fan $\Sigma$ in $N\otimes_\Z\R\cong\R^n$ where $N$ is a 
lattice of rank $n$. We denote the dual lattice Hom${}_\Z(N)$ with $M$. Then $T=\mbox{Spec } \C(M)\cong (\C^\ast)^n$ is a 
dense open torus acting on $X$. We will assume that $X$ is complete, i.e., $\Sigma$ has all of $N\otimes_\Z\R$ as support, or even that 
$X$ is smooth which in terms of $\Sigma$ is equivalent to  the following: $\Sigma$ is simplicial and for each  maximal (i.e. $n$-dimensional) cone $\sigma\in\Sigma$ the primitive generators $\v_i$ of the rays $\rho_i\in\sigma(1)$, for $i=1,\dots,n$,  form a lattice basis of $N$. 

Toric varieties are particularly transparent from the viewpoint of birational geometry as their minimal model programs can be completely
described in terms of convex geometry, i.e., by considering fans: Let $\Sigma$ be a simplicial cone. Then an extremal ray $\mathcal R$
in the Mori cone $\NE(X)$ induces a toric morphism (an \emph{extremal contraction}) $\phi:X\to X'$ to a semi-projective 
toric variety such that for a wall 
$\tau\in\Sigma(n-1)$ the corresponding curve $V(\tau)$ is contracted to a point if and only if the class of $V(\tau)$ lies in $\mathcal R$ 
(\cite[Proposition 15.4.1]{CLS}). 
In order to describe $\phi_{\mathcal R}$, remember that an extremal ray $\mathcal R$ is generated by an  extremal relation 
$$
	b_{\rho_1}\v_{\rho_1}+\dots b_{\rho_{n+1}}\v_{\rho_{n+1}} = 0
$$
where the $\v_{\rho_i}$ are the primitive generators of the $n+1$ rays which span those two maximal cones $\sigma$, $\sigma'$ whose 
intersection is the wall $\tau$ corresponding to an invariant curve in $\mathcal R$.  Setting
\begin{align*}
	J_-&=  J_{\mathcal R, -} = \set{\rho\in \Sigma(1)\mid b_\rho <0}\\
	J_+&= J_{\mathcal R, +} = \set{\rho\in \Sigma(1)\mid b_\rho >0},
\end{align*}
we have the following description of the toric extremal contraction $\phi_{\mathcal R}$.
\begin{proposition}[\cite{CLS}, Lemma 15.5.2, Proposition 15.4.5] \label{Th:cls} 
	The exceptional locus of $\phi_{\mathcal R}$ has codimension $|J_-|$ and the fiber over points in 
	its image are fake weighted projective spaces  of dimension $|J_+|-1$, concretely,
	\begin{enumerate}
		\item $\phi_{\mathcal R}$ is of fiber type if and only if $J_-=\emptyset$.
		\item $\phi_{\mathcal R}$ is divisorial if and only if $|J_-|= 1$. In this case the image of the exceptional divisor has dimension
		$n - |J_+|$.
		\item $\phi_{\mathcal R}$ is a small contraction if and only if $|J_-|>1$.
	\end{enumerate}
\end{proposition}

\begin{remark}\label{Rmk:smooth}
	The wall relation $b_{\rho_1}\v_{\rho_1}+\dots b_{\rho_{n+1}}\v_{\rho_{n+1}} = 0$ corresponding to a wall $\tau$ spanned by $\v_2,\dots,\v_n$ can also 
	be realized as a positive multiple of the linear relation
	$$
		\sum_{\rho}(D_\rho\cdot V(\tau))\v_\rho =0,
	$$
	(see \cite[Section 6.4]{CLS}).  By \cite[Proposition 6.4.4]{CLS}, the intersection numbers in this relation can be expressed as 
	$$
		D_{\rho_i}\cdot  V(\tau) = \frac{b_i\mult(\tau)}{b_1\mult(\sigma)} =  \frac{b_i\mult(\tau)}{b_{n+1}\mult(\sigma')}.
	$$
	Here the multiplicity $\mult(\sigma)$ of a simplicial cone is the index in $N_\sigma=\sigma\cap N +(-\sigma)\cap N$  of the 
	lattice generated by 
	the primitive generators $\v_1\dots,\v_k$ of $\sigma$. Hence, in case $X$ is smooth, all multiplicities are one, 
	so in the wall relation we can assume $b_1=b_{n+1}=1$. 
\end{remark}
Another result we will use gives estimates for the length $\ell(\mathcal R)$ of an extremal ray $\mathcal R$, i.e., the smallest intersection
of $-K_X$ with an irreducible curve $C$ with class in $\mathcal R$. Mori's  bend and break technique produces for any $(K_X)$-negative extremal ray a possibly singular rational curve $C$ with class in $\mathcal R$ satisfying $0<-K_XC\le\dim X+1$. In the toric situation 
the upper bound has recently been improved by Fujino and Sato with the additional benefit that the curve in which the upper 
bound  is realized is smooth and torus-invariant at least in the case where the associated map $\phi_{\mathcal R}$ is birational. 

\begin{theorem}[\cite{FS2}, Theorem 3.2.1]\label{Th:FS}
	Let $X$ be a $\Q$-factorial  toric variety of dimension $n$. Let $\phi_{\mathcal R}: X\to W$ be a birational contraction 
	associated to a $(K_X)$-negative extremal ray $\mathcal R$ in $\NE(X)$. Denoting 
	$$
		d=\max_{w\in W} \dim \phi_{\mathcal R}^{-1}(w),
	$$ 
	there exists a torus-invariant curve $C$ with $[C]\in\mathcal R$ such that
	$$
		-K_X\cdot C< d+1 \le n.
	$$
	If $\phi_{\mathcal R}$ is a  divisorial contraction, i.e., $d=n-1$, then we get the stronger inequality $\ell(\mathcal R)\le n-1$ 
	and equality holds if only if
	$\phi_{\mathcal R}$ is a weighted blowup of a smooth torus-invariant point with weight $(1,a,\dots,a)$ for 
	some positive integer $a$. 
\end{theorem}
\begin{remark}
	Fujino and Sato do not state the torus-invariance of the curve $C$ realizing the bound in the result. However,  it is apparent 
	from their proof. In particular, if $\phi_{\mathcal R}$ is birational, and the 
	contraction corresponds to  walls $\tau_i\in\Sigma(n-1)$, then for some $i$ we have $-K_X\cdot V(\tau_i)<d+1$. 
	
	Furthermore, in the smooth case the upper bound is  realized in the invariant 
	curve $C=V(\tau)$ for 
	$\tau$ the face on $\Sigma$ corresponding to the ray $R$. This follows from 
	Remark \ref{Rmk:smooth}, as the identity $D_{\rho_n}V(\tau)=1$ for the effective
	 integral divisor $D_{\rho_n}$ rules out the existence of an invariant curve $C'$ in $R$ such that
	 $C=kC'$ for $1<k\in\Z$, and the fact that $-K_X\cdot C>0$.
\end{remark}


\section{Exterior powers of the tangent bundle}
Let $X_\Sigma$ be a non-singular toric variety corresponding to a fan $\Sigma$ 
whose rays $\rho_1,\dots,\rho_d$ are generated minimally by lattice points $v_1,\dots,v_d$. 

The Klyachko filtration of the tangent bundle on $X$ is given by 
\begin{alignat*}{2}
	E^{\rho_j}(i) = \begin{cases}E&\qquad i\le 0 \\
						<v_j>&\qquad i=1\\
						0 	&\qquad i>1,
					\end{cases}
\end{alignat*}
(see \cite[Subsection 2.3.5]{K}). Gonz\'alez in \cite[Section 3]{G11} calculates the Klyachko filtrations of symmetric 
and exterior powers of an equivariant 
 vector bundle $E$  from the filtration of $E$ itself. However, for our purposes the following description of 
 the tangent bundle itself suffices to deduce positivity criteria for its exterior powers.

By Theorem \ref{Th:TestOnCurves}, in order to determine nefness or ampleness of an equivariant bundle,
we need to study its restrictions to invariant curves on $X$. In the case of the tangent bundle this is particularly 
simple and closely related to the possible toric extremal contractions of $X$ as we will see now. Note that the dual 
calculation for the cotangent bundle  $\Omega_X$ has already been described in \cite[Example 5.2]{DJS}. 
For the convenience of the reader we will repeat the steps of the calculation for $\T_X$ here. Let $C$ be an invariant 
curve on $X$ corresponding to an $(n-1)$-dimensional face (or \emph{wall}) $\tau$ of the fan $\Sigma$. 
Let furthermore $\sigma$ and $\sigma'$ be the two full-dimensional cones adjacent along $\tau$. Since $X$ is smooth,  there is a lattice 
basis $\v_1,\dots,\v_n$ such that $\sigma=\pos(\v_1,\dots,\v_n)$. By  Remark \ref{Rmk:smooth}, we may assume that $\sigma'=\pos(\v_1,\dots,\v_{n-1},\v_{n+1})$ where $\v_{n+1}=a_1\v_1+a_2\v_2+\dots+a_{n-1}\v_{n-1}-\v_n$ for some $a_j\in\Z$. 
The restrictions $\T_X|_{U_\sigma}$ and $\T_X|_{U_{\sigma'}}$ decompose as sums of line bundles
$$
	\T_X|_{U_\sigma} = \bigoplus_{i=1}^n \mathcal L_{\w_i}, \quad \T_X|_{U_{\sigma'}} = \bigoplus_{i=1}^n \mathcal L_{\widetilde\w_i},
$$
where $\w_1,\dots,\w_n$ is the dual basis to $\v_1,\dots,\v_n$, and $\widetilde\w_1,\dots,\widetilde\w_n$ is the dual basis to $\v_1,\dots,\v_{n-1},\v_{n+1}$. Expressing the latter in terms of the 
former, we obtain associated characters  ${\bf u}(\sigma')=\set{\w_1+a_1\w_n,\w_2+a_2\w_n,\dots, \w_{n-1}+a_{n-1}\w_n,-\w_n}$, and the characters are paired along $C$ as follows: $(\w_n,-\w_n)$ and $(\w_i,\w_i+a_i\w_n)$ for $1\le i\le n-1$. Since  $\w_\tau = \w_n$, and $\w_i-(\w_i+a_i\w_n) = -a_i\w_n$, the restriction of $\mathcal T_X$ to the invariant curve $C\cong \P^1$ decomposes as a sum of line bundles
$\T_X|_C\cong \O_{\P^1}(2) \oplus \O_{\P^1}(-a_1)\oplus\dots\oplus  \O_{\P^1}(-a_{n-1})$. 

Note that calculating the numbers $a_i$ above is equivalent to determining the wall relation 
$$
	\sum_{i=1}^{n+1} b_i\v_i = 0
$$
corresponding to the wall $\tau$. On the other hand, as we have seen in Section \ref{s:contractions}, 
in case $C$ is an extremal curve the vector $(b_1,\dots,b_{n+1})$ yields plenty
of information about the birational map
given by contracting $C$. We record the above result in 
\begin{proposition}
	Let $C$ be an invariant curve on $X_\Sigma$ corresponding to a face $\tau\in\Sigma(n-1)$ and let
	$b_1\v_1+\dots+b_{n-1}\v_{n-1}+\v_n +\v_{n+1}=0$ be the corresponding wall relation. Then
	$$
		\T_X|_C\cong  \O_{\P^1}(2) \oplus \O_{\P^1}(b_1)\oplus\dots\oplus  \O_{\P^1}(b_{n-1}).
	$$ 
\end{proposition} 

The above proposition together with Theorem \ref{Th:TestOnCurves} gives us  very tangible criteria for positivity of exterior powers
of the tangent bundle:
\begin{theorem}\label{Th:Criterion}
	 Let $X_\Sigma$ be a smooth toric variety of dimension $n$. Then for $1\le m\le n-1$ the exterior power 
	 $\Lambda^m\T_X$ is ample (nef) if and only if
	 for any $\tau\in\Sigma(n-1)$ with wall relation $b_1\v_1+\dots+b_{n-1}\v_{n-1}+\v_n +\v_{n+1}=0$ the inequalities
	 \begin{align*}
	 	b_{i_1}+b_{i_2}+\dots + b_{i_m} &\geoq 0\\
		2 + b_{j_1}+b_{j_2}+\dots + b_{j_{m-1}} &\geoq 0
	 \end{align*}
	 hold for each $1\le i_1<\dots<i_m\le n$ and $1\le j_1<\dots<j_{m-1} \le n$.\\
	 Similarly,  $\Lambda^n\T_X=-K_X$ is ample (nef) if and only if for any $\tau\in\Sigma(n-1)$ the coefficients 
	 in the corresponding wall relation satisfy
	 $$
	 	2 + b_{1}+b_{2}+\dots + b_{{n-1}} \geoq 0.
	 $$
	 
\end{theorem}

\begin{remark}\label{rm:TXnef}
	(1) In  case $m=1$ the criterion yields that  $\T_X$ is ample if and only if all $b_i$ are positive or, in the terminology of 
	\cite{CLS}, $J_-=\emptyset, J=J_+$ for each invariant extremal curve $C$.  By Proposition \ref{Th:cls},  this is the case if and 
	only it the extremal contraction corresponding to $C$ is of fiber type with image of dimension zero. Thus we recover 
	the well known result that $X=\P^n$. 
	
	(2) On the other hand, for $\T_X$ to be nef, it suffices (and is  necessary) for all $b_i$ in wall relations to be non-negative. As the Mori cone cone $\NE(X)$ is spanned by extremal curves, this is equivalent to all extremal contractions being of fiber type. In fact, it is known in general that a non-singular variety $X$ with 
	$\T_X$ nef  admits only fiber type extremal contractions, since the deformations of any rational 
	curve cover all of $X$ in this case (see \cite[Theorem 2.2]{CP91}). In general, the inverse 
	implication is false (take for example the product $\P^1\times X$ for a non-singular variety $X$ of Picard number 1 with
	$\T_X$ not nef). 
	
	A result  by Fujino and Sato (\cite[Proposition 5.3]{FS1}) states that a smooth toric 
	variety admitting only fiber type contractions is isomorphic to a product of projective 
	spaces, providing a full classification of smooth toric varieties with nef tangent bundle.
\end{remark}
\begin{example} Let $X$ be the blowup of $\P^3$ in a line, i.e., the toric variety corresponding to the 
 fan in $N=\Z^3$ obtained as follows. Take the standard fan generated by the points $\v_1=(1,0,0)$, $\v_2=(0,1,0)$, $\v_3=(0,0,1)$, and $\v_4=(-1,-1,-1)$, corresponding to $\P^3$ and add the ray $\v_5=(-1,-1,0)$  subdividing the maximal cone $<\v_1,\v_3,\v_4>$ by the two maximal cones $<\v_1,\v_3,\v_5>$, and $<\v_1,\v_4,\v_5>$ and the maximal cone $<\v_2,\v_3,\v_4>$ by the two maximal cones $<\v_2,\v_3,\v_5>$, and $<\v_2,\v_4,\v_5>$ . We then get the following wall relations (and equivalent ones by symmetry)
 \begin{align*} 
	\v_3+\v_4+\v_1+\v_2 &= 0\qquad\mbox{(not extremal)}\\
	\v_1+\v_5+\v_2+0\cdot\v_3&=0\qquad\mbox{(pencil of planes through blown-up line)}\\
 	\v_3+\v_4-\v_5+0\cdot\v_1 &= 0\qquad\mbox{(blow-up morphism)}.\\
\end{align*}
Hence, Theorem \ref{Th:Criterion} yields that $\T_X$ and $\Lambda^2\T_X$ are neither ample nor nef, whereas $\Lambda^3\T_X=-K_X$ is  ample.

This example can be extended to blow-ups $\phi:X\to \P^n$ of a $k$-dimensional linear
 subspace. As above, the only relation containing a negative coefficient corresponds to the blow-up morphism and it has exactly $k$ zero entries. Hence $\Lambda^m\T_X$ is nef if and only $m$ is at least $k+2$.

\end{example}
Let us state an easy consequence of the ampleness/nefness criterion, which is not at all obvious in general. 
\begin{cor}
	Let $X$ be a non-singular toric variety with $\Lambda^m\T_X$ nef for some $m<n$. Then $X$ is Fano.
\end{cor}
\begin{proof}
	Let $C=V(\tau)$ be an invariant curve with corresponding extremal relation  
	$b_1\v_1+\dots+b_{n-1}\v_{n-1}+\v_n +\v_{n+1}=0$. We may assume the $b_i$ to be ordered such that 
	$b_i \le b_{i+1}$ for $i=1,\dots,n-2$. Then by the nefness of $\Lambda^m\T_X$ we have
	$$
		-K_X\cdot C = 2 + \sum_{i=1}^{n-1} b_i > \sum_{i=1}^m b_i \ge 0.
	$$ 
\end{proof}

\section{Classification}

In order to  classify smooth  toric varieties with $\Lambda^2\T_X$ nef (or $\Lambda^3\T_X$ ample), we will 
argue inductively by applying toric extremal contractions, thus decreasing the Picard number. We show that by this process  we 
arrive eventually at a smooth toric variety $Y$ which admits only fiber type extremal contractions. 
By Remark \ref{rm:TXnef}, this is equivalent to  $T_Y$ being nef which in turn means that $Y$ is a  product of 
projective spaces.

\begin{remark}
The case of Picard number one is easy in the toric situation as opposed to the general setting as the only 
smooth toric varieties with Picard number one are projective spaces. This can either be seen directly by fan considerations 
(see \cite[Exercise 7.3.10]{CLS}), or we might use the following argument.
If a non-singular projective toric variety $X$ has Picard number $b_2(X)=1$, then  the dual Euler sequence
$$
	0 \to \mbox{Pic}(X)\otimes \O_X \to \bigoplus_\rho \O_X(D_\rho) \to \T_X \to 0 
$$
(see \cite[Theorem 8.1.6]{CLS}) realizes $\T_X$ as a quotient of the ample vector bundle $\bigoplus\O_X(D_\rho)$ 
proving it is ample. Therefore, $X\cong \P^n$ by Mori's famous result.  
\end{remark}

\subsection{$\Lambda^2\T_X$ nef}
The elementary result for the classification is the following
\begin{theorem}\label{th:L2} 
	Let $f:X\to Y$ be an extremal contraction of a nonsingular toric variety of dimension $n\ge3$ 
	with $\Lambda^2\T_X$ nef. 
	Then either $f$ is of fiber type, or $f$ is a  divisorial 
	contraction whose exceptional locus has zero-dimensional image and $Y$ is smooth with $\Lambda^2\T_Y$ nef.
	In the latter case, $f$ consists of blowups of invariant points. 
\end{theorem}
\begin{proof}
	 We may assume $f$ to be of relative Picard number one. 
	 Let $b_1\v_1+\dots+b_{n-1}\v_{n-1}+\v_n +\v_{n+1}=0$  be the extremal relation corresponding 
	 to $f$, ordered such that $b_i\le b_{i+1}$ for $i=1,\dots,n-2$. Then by Theorem \ref{Th:Criterion}  
	 we have
	 $$
	 	b_1+b_2\ge0, \qquad b_1 \ge -2.
	 $$
	 Therefore, either $b_1$ and hence all $b_i$ are non-negative and $f$ is of fiber type, or 
	 $b_1$ is negative, and the remaining $b_i$ are non-negative, i.e., $J_0=\emptyset$. By 
	 Proposition \ref{Th:cls}, in the latter case $f$ is a divisorial contraction with 
	 image of the exceptional locus of dimension zero.
	 
	 Why is $Y$ smooth in this case? If $b_1=1$, then
	 $\v_2,\dots\v_n,\v_{n+1}$ form a lattice  basis. Therefore, $Y$ is nonsingular. 
	 If on the other hand $b_1=-2$ then all other $b_i$ are at least $2$, so 
	 $-K_X\cdot C\ge 2(n-2)\ge n-1$, which is impossible by Theorem \ref{Th:FS}. 
	 In fact, the only possibility left by this result  is $b_1=-1$, and $b_i=1$ for $i=2,\dots, n-1$, i.e.,
	 $-K_X\cdot C=n-1$. Hence, again by Theorem \ref{Th:FS}, $f$ is a weighted blowup of a smooth invariant
	 point of weight $(1,a,\dots,a)$. Since $X$ and $Y$ are smooth, $a=1$. 
	 
	 By the relative tangent sequence for $f$, we get that $f^\ast(\Lambda^2\T_Y)$ and thus $\Lambda^2\T_Y$ itself 
	 are nef. 
\end{proof}
\noindent
We are thus in the situation to argue inductively and obtain the following classification.

\begin{theorem}\label{th:class} 
	Let $X$ be a smooth toric variety of dimension $n\ge3$ with $\Lambda^2\T_X$ nef. Then either $\T_X$ is nef, or 
	$X$ is the blowup of $P^n$ in a point. 
\end{theorem}

\begin{proof}
	If $X$ has Picard number one, then $X=\P^n$, hence $\T_X$ is nef. 
	Otherwise $X$ admits  toric extremal contractions $f_i:X\to Y_i$. 
	If these are all of fiber type, then by Remark \ref{rm:TXnef} 
	$\T_X$ is nef.
	
	 Hence, if $\T_X$ is not nef, there is a birational  extremal contraction $f:X\to Y$, which by 
	 Theorem \ref{th:L2} must be the blowup of an invariant point $p\in Y$. 
	 
	 Performing all such birational extremal contractions, we obtain a toric variety $Y_0$ 
	 all of whose extremal contractions are fiber type. Therefore, $Y_0=\P^{k_1}\times\dots\times\P^{k_s}$ 
	 for $k_1+\dots+k_s=n$. Now, $X$ does not admit any flipping contractions. However, blowing up a point on 
	 an $n$-dimensional proper product of projective spaces will lead to flips. Hence, $Y_0=\P^n$. 
	 Similarly, blowing up more than one point on $\P^n$ again gives rise to flips.
\end{proof}

\subsection{$\Lambda^3\T_X$ ample}

We get analogous results in the situation when $\Lambda^3\T_X$ is ample:
\begin{theorem}\label{th:L3} 
	Let $f:X\to Y$ be an extremal contraction of a nonsingular toric variety of dimension $n\ge4$ 
	with $\Lambda^3\T_X$ ample. 	Then either $f$ is of fiber type, or $f$ is a  divisorial 
	contraction whose exceptional locus has zero-dimensional image and $Y$ is smooth with $\Lambda^3\T_Y$ ample.
	In the latter case, $f$ consists of blowups of invariant points. 
\end{theorem}
\begin{proof}
	 This is a  variation of the proof of Theorem \ref{th:L2}.
	 We may assume $f$ to be of relative Picard number one. 
	 Let $b_1\v_1+\dots+b_{n-1}\v_{n-1}+\v_n +\v_{n+1}=0$  be the extremal relation corresponding 
	 to $f$, ordered such that $b_i\le b_{i+1}$ for $i=1,\dots,n-2$. Then by Theorem \ref{Th:Criterion}  
	 we have
	 $$
	 	b_1+b_2 +b_3 > 0, \qquad b_1 + b_2 > -2.
	 $$
	 By the second inequality, the only possibly negative $b_i$ is $b_1$. Therefore,  $X$ does not admit flipping
	 contractions. Furthermore, if $b_1<0$, then either $b_2$, and hence all remaining $b_i$, are positive, or $b_2=0$ and 
	 $b_3>-b_1$.  However, in the latter case we have $-K_X\cdot C\ge 2+b_2+(n-3)(-b_2+1) = 2 +1 +(n-4)(-b_1+1) \ge n-1$
	 with equality if and only if $n=4$. This is impossible by Theorem \ref{Th:FS}. Thus if $b_1<0$, then  $J_0=\emptyset$ and 
	 $f$ is thus a divisorial contraction with image of the exceptional locus of dimension zero. 
	 By the same argument as in the proof of Theorem \ref{th:L2}, $f$ must in fact be the blowup of a smooth invariant point. 
\end{proof}

With the above result established, the  proof of Theorem \ref{th:class},
\emph{mutatis mutandis}, proves the following.

\begin{theorem}
	Let $X$ be a smooth toric variety of dimension $n\ge4$ with $\Lambda^3\T_X$ ample. 
	Then either $\T_X$ is nef, or 
	$X$ is the blowup of $P^n$ in a point. 
\end{theorem}

\subsection{Higher exterior powers}
Weakening the assumptions to positivity of  higher exterior powers of $\T_X$, Theorem \ref{Th:Criterion} leads to weaker 
inequalities regarding
the $b_i$ in extremal relations. We thus get fewer restrictions on the possible toric birational morphisms $X\to W$. For
example, assuming $\Lambda^3\T_X$ nef, the extremal relation with $\vec{b}=(-2,1,\dots,1)$ is permitted by the inequalities
of Theorem \ref{Th:Criterion}. On the other hand, its corresponding curve $C$ has $-K_XC=n-2$ which does not violate  
Theorem \ref{Th:FS}. Therefore, we cannot exclude the possibility that some contractions might have singular image $Y$ in this situation. 
In particular, $\T_Y$ will not be locally free if this happens.
Dealing with isolated singularities arising in this way still is possible with an inductive approach. However, taking even higher exterior 
powers forces us to deal with flipping contractions. For these cases, which arise  starting at $\Lambda^4\T_X$ nef (e.g., for $\vec{b}
=(-1,-1,1,\dots, 1))$, a different  approach will be necessary as $\Sigma_Y$ will not even be simplicial.




\footnotesize

   \bigskip
   David Schmitz,
   Mathematisches Institut,
   Universität Bayreuth,
  D-95440 Bayreuth
  
   \nopagebreak
   \textit{E-mail address:} \texttt{schmitzd@mathematik.uni-marburg.de}


\end{document}